\documentclass[10pt]{amsart}

\usepackage[top=3cm,bottom=3cm]{geometry}

\newcommand{\Z}{\mathbb{Z}}

\newcommand{\N}{\mathbb{N}}

\newcommand{\pr}{\mathbb{P}}
\newcommand{\F}{\mathscr{F}}
\newcommand{\G}{\mathscr{G}}
\newcommand{\Ol}{\mathcal{O}}

\usepackage{amsmath}
\usepackage{amssymb, amsfonts,amsthm,amscd}
\usepackage{latexsym,amsthm,amscd,flafter,graphicx,verbatim,stmaryrd,mathrsfs,color}
\usepackage[all]{xy}
\usepackage{mathrsfs}
\usepackage{nicefrac} 
\usepackage{xfrac}	%
\usepackage{faktor} 	%
\usepackage{lscape}
\usepackage[strict]{changepage}
\usepackage[tiling]{pst-fill}

\begin{document}

\newtheorem{theorem}{Theorem}[section]
\newtheorem{lemma}[theorem]{Lemma}
\newtheorem{definition}[theorem]{Definition}
\newtheorem{proposition}[theorem]{Proposition}
\newtheorem{corollary}[theorem]{Corollary}
\newtheorem{remark}[theorem]{Remark}

\title{On the existence of Fourier-Mukai kernels}
\author{\bf{Alice Rizzardo} }

\begin{abstract}A theorem by Orlov states that any equivalence $F:D^{b}_{\mathrm{Coh}}(X) \to D^{b}_{\mathrm{Coh}}(Y)$ between the bounded derived categories of coherent sheaves of two smooth projective varieties $X$ and $Y$
is isomorphic to a Fourier-Mukai transform $\Phi_{E}(-)=Rp_{2*}(E\stackrel{L}{\otimes} Lp_{1}^{*}(-))$, where the kernel $E$ is in  $D^{b}_{\mathrm{Coh}}(X\times Y)$. In the case of an exact functor which is not necessarily fully faithful, we compute some sheaves that play the role of the cohomology sheaves of the kernel, and that are isomorphic to the latter whenever an isomorphism $F\cong\Phi_{E}$ exists. We then exhibit a class of functors that are not full or faithful and still satisfy the above result.
\end{abstract}

\maketitle

\section{Introduction}
Let $X$, $Y$ be smooth projective varieties over an algebraically closed field $k$. Consider an exact functor 
$$F:D^{b}_{\mathrm{Coh}}(X) \to D^{b}_{\mathrm{Coh}}(Y)\mathrm{.}$$
Orlov proved in \cite{orlov} that, if $F$ is a fully faithful functor and has a right adjoint, then there exists an object $E\in D^{b}_{\mathrm{Coh}}(X\times Y)$ such that $X$ is isomorphic to the Fourier-Mukai transform $\Phi_{E}$, defined as
$$\Phi_{E}(-)=Rp_{2*}(E\stackrel{L}{\otimes} Lp_{1}^{*}(-))\mathrm{.}$$

The requirement that $F$ needs to have a right adjoint is actually unnecessary, since by \cite[Theorem 1.1]{generators}, every exact functor $F:D^{b}_{\mathrm{Coh}}(X) \to D^{b}_{\mathrm{Coh}}(Y)$ has a left and a right adjoint. (For an explanation on why this is true, see for example \cite[Remark 2.1]{twisted}.)

There is evidence that this theorem should generalize to the case where the functor is not full or faithful. Canonaco and Stellari, in their paper \cite{twisted}, partially following Orlov's proof, give a weaker condition for the functor for it to be still isomorphic to a Fourier-Mukai transform. The condition is in particular satisfied if the functor is full. This is however a bittersweet result: in fact, in \cite{full} together with Orlov they also proved that when $X$ and $Y$ are smooth projective varieties over a field of characteristic 0, if a functor is full then it is also faithful.

Although we are not able to prove that an isomorphism exists in general, given any exact functor as above, we make some progress towards identifying the kernel of the Fourier-Mukai functor, if it exists - namely, we can explicitly compute the kernel's cohomology sheaves, which we will denote by $\mathscr{Bla}^{i}$:

\begin{theorem}\label{all blas}
Let $X$, $Y$ be smooth projective varieties over an algebraically closed field $k$, and $F:D^{b}_{\mathrm{Coh}}(X) \to D^{b}_{\mathrm{Coh}}(Y)$ an exact functor. There exist sheaves $\mathscr{Bla}^{i}\in\mathrm{Coh}(X\times Y)$, with $i \in \mathbb{Z}$, all but a finite number of them equal to zero, such that for every locally free sheaf of finite rank $\mathscr{E}$ on $X$ and for all $n\in \mathbb{Z}$ there are functorial maps
$$ \mathscr{H}^{i}(F(\mathscr{E}(n)))\to p_{2*}(\mathscr{Bla}^{i}\otimes p_{1}^{*}\mathscr{E}(n))$$
which are isomorphisms for $n$ sufficiently high (depending on $\mathscr{E}$).

In particular, if $F\cong \Phi_{E}$, i.e. $F$ is the Fourier-Mukai transform with kernel $E$, then $\mathscr{H}^{i}(E)\cong \mathscr{B}^{i}$ for all $i\in\mathbb{Z}$.
\end{theorem}
%

%

We are then able, in a special case, to construct an isomorphism between a class of functors that are not full or faithful and a Fourier-Mukai transform:
\begin{theorem}\label{generaliso}
Let $X$ and $Y$ be smooth projective varieties over an algebraically closed field $k$, with $X$ of dimension one,  and $F:D^{b}_{\mathrm{Coh}}(X)\to D^{b}_{\mathrm{Coh}}(Y)$ an exact functor. Assume that there exists an integer $M$ such that the sheaves $\mathscr{Bla}^{i}$ associated to $F$ as in Theorem \ref{all blas} are zero  for $i \neq M$, and that $\mathscr{Bla}^{M}$ is a direct sum of skyscraper sheaves $\mathscr{Bla}^{M}=\bigoplus_{i=1}^{t} k(p_{i},q_{i})$. Let $\Phi$ be the Fourier-Mukai transform associated to the complex given by the sheaf $\mathscr{Bla}^{M}$ placed in degree $M$, $\Phi=\Phi_{\mathscr{Bla}^{M}[-M]}$. Then there exists an isomorphism of functors $s:  \Phi \to F$.
\end{theorem}

Even when we don't know how to build a kernel out of the sheaves $\mathscr{Bla}^{i}$ obtained in Theorem \ref{all blas}, these sheaves turn out to have good properties in their own right. As an example, in section \ref{ss} we show that the analogue of the Cartan-Eilenberg Spectral Sequence converges when the dimension of $X$ is one.

\subsection*{Notation}From now on, $X$ and $Y$ will be smooth projective varieties over an algebraically closed field $k$; $p_{1}$ and $p_{2}$ will be the projections $X\times Y\to X$ and $X\times Y\to Y$ respectively; $F:D^{b}_{\mathrm{Coh}}(X) \to D^{b}_{\mathrm{Coh}}(Y)$ will be an exact functor, i.e. an additive functor that commutes with shifts and preserves triangles; and $\Ol_{X}(1)$ will be a very ample line bundle on $X$. 

Given a morphism $p$, we will use $p_{*}$ and $p^{*}$ to indicate pushforward and pullback on coherent sheaves, whereas pullback and pushforward in the derived category will always be indicated with $Lp^{*}$ and $Rp_{*}$ unless they coincide with the regular pullback and pushforward, in which case both notations will be used interchangeably.  

For a smooth variety $X$, $\mathrm{Coh}(X)$ will be considered as a full subcategory of $D^{b}_{\mathrm{Coh}}(X)$ by associating to a sheaf the complex given by that sheaf placed in degree zero. Hence we will write $F(\mathscr{E})$ to indicate $F(\mathscr{E}[0])$.

\subsection*{Acknowledgements}This paper is derived from part of the author's PhD thesis. The author thanks her thesis advisor Aise Johan de Jong for suggesting the problem as well as providing invaluable guidance over the years. The last draft of this paper was written while at SISSA, in Trieste.

\section{Determining the cohomology sheaves of the prospective kernel}

Consider an exact functor $F:D^{b}_{\mathrm{Coh}}(X) \to D^{b}_{\mathrm{Coh}}(Y)$. If we know that $F$ is isomorphic to a Fourier-Mukai transform $\Phi_{E}$, then we are of course able to compute the cohomology sheaves $\mathscr{Bla}^{i}=\mathscr{H}^{i}(E)$ corresponding to $E$. Even if we don't know what $E$ is, or even if it exists, we are able to compute some sheaves on $X\times Y$ that, if the functor comes from a Fourier-Mukai transform, turn out to be the cohomology sheaves of the corresponding kernel. 


\begin{lemma}\label{bla_exists}
Let $X$, $Y$ be smooth projective varieties over an algebraically closed field, $\mathscr{O}_{X}(1)$ a very ample invertible sheaf on $X$. There exists an equivalence of categories between the category of coherent sheaves on $X\times Y$ and the category of graded coherent $\Gamma_{*}(\Ol_{X})\otimes \Ol_{Y}$-modules $\mathscr{M}=\oplus_{k}\mathscr{M}_{k}$ such that $\oplus_{k\geq n} \mathscr{M}_{k}$ is finitely generated for some n, where two coherent sheaves are identified if they agree in sufficiently high degree.

Moreover, if this correspondence associates a sheaf $\oplus \mathscr{M}_{n}$ on $Y$ to a sheaf  $\mathscr{Bla}$ on $X\times Y$, there exists a functorial map of graded $\Gamma_{*}(\Ol_{X})\otimes \Ol_{Y}$-modules
$$\oplus \mathscr{M}_{n}\to \oplus p_{2*}(\mathscr{Bla}\otimes \mathscr{O}_{X}(n))$$
which is an isomorphism on the $n^{th}$ graded piece for $n$ sufficiently high.
\end{lemma}

\begin{proof}
This is a routine application of \cite[3.2.4, 3.3.5, 3.4.3, 3.4.5, 3.3.5.1]{EGA2}. The fact that we get an isomorphism $\oplus \mathscr{M}_{n}\to \oplus p_{2*}(\mathscr{Bla}\otimes \mathscr{O}_{X}(n))$ in large enough degree can be checked locally and hence follows by  \cite[\S 65, Proposition 5]{FAC}.
\end{proof}

\begin{proof}[proof of Theorem \ref{all blas}]
By \cite[Lemma 2.4]{orlov}, we can assume that $F$ is bounded, i.e. that $F(\mathscr{F})\in D^{[a,b]}_{\mathrm{Coh}}(Y)$ for all coherent sheaves $\mathscr{F}$ on $X$, so that $\mathscr{H}^{i}(F(\mathscr{E}))=0$ for $i\notin [a,b]$. It follows immediately that there are integers $M$ and $N$ such that $F(\mathscr{E})\in D^{[M,N]}_{\mathrm{Coh}}(Y)$ for all locally free sheaves $F(\mathscr{E})$ of finite rank on $X$. Therefore we can take $\mathscr{Bla}^{i}=0$ for $i\notin [M,N]$.

We will proceed by descending induction on the cohomology degree $i$. Assume we found the sheaves $\mathscr{Bla}^{N},\mathscr{Bla}^{N-1},\ldots,\mathscr{Bla}^{i+1}$ satisfying the conclusions of the Theorem and let us compute the sheaf $\mathscr{Bla}^{i}$. To do this we will proceed in two steps: first we will construct sheaves $\mathscr{Bla}^{i}_{\mathscr{E}}$ for all locally free sheaves $\mathscr{E}$ of finite rank, as well as maps
$$ \mathscr{H}^{i}(F(\mathscr{E}(n)))\to p_{2*}(\mathscr{Bla}^{i}_{\mathscr{E}}\otimes p_{1}^{*}\Ol_{X}(n))$$
that are isomorphisms for $n$ sufficiently high, depending on $\mathscr{E}$ and $i$. Then we will show that 
$$\mathscr{Bla}^{i}_{\mathscr{E}}=\mathscr{Bla}^{i}_{\Ol_{X}}\otimes p_{1}^{*}\mathscr{E}\mathrm{.}$$

For the first step, the key is showing that the sheaf $\bigoplus_{n>n_{0}} \mathscr{H}^{i}(F(\mathscr{E}(n)))$ on $Y$ is finitely generated for each $n_{0}$ as a $\Gamma_{*}(X,\Ol_{X})\otimes \Ol_{Y}$-module. To do this, proceed as follows: let $s$ be an integer such that we have a surjection $ \Ol_{X}^{\oplus s}\to \Ol_{X}(1)$. Let $\mathscr{E}$ be a locally free sheaf of finite rank on $X$. Then by tensoring the map above with $\mathscr{E}$ and twisting by $n$ we have a short exact sequence of locally free sheaves
$$0\to K(n)\to \mathscr{E}^{\oplus s}(n)\to \mathscr{E}(n+1)\to 0\mathrm{.}$$
Hence
$$0\to p_{1}^{*}(K(n)) \to p_{1}^{*}(\mathscr{E}(n)^{\oplus s})\to p_{1}^{*}(\mathscr{E}(n+1))\to 0$$
is also a short exact sequence of locally free sheaves, and tensoring with $\mathscr{Bla}^{i+1}$ will yield another short exact sequence:
$$0\to \mathscr{Bla}^{i+1}\otimes p_{1}^{*}K(n) \to \mathscr{Bla}^{i+1}\otimes p_{1}^{*}\mathscr{E}^{\oplus s}\to \mathscr{Bla}^{i+1}\otimes p_{1}^{*}\mathscr{E}(n+1)\to 0\mathrm{.}$$
Moreover, since $p_{1}^{*}\Ol_{X}$ is very ample with respect to $X\times Y\to Y$, for $n$ high enough (depending on $K$) the pushforward to $Y$ will still be exact:
$$0\to p_{2*}(\mathscr{Bla}^{i+1}\otimes p_{1}^{*}K(n)) \to p_{2*}(\mathscr{Bla}^{i+1}\otimes p_{1}^{*}\Ol_{X}(n)^{\oplus s})\to p_{2*}(\mathscr{Bla}^{i+1}\otimes p_{1}^{*}\Ol_{X}(n+1))\to 0\mathrm{.}$$
Hence we get a commutative diagram
$$\xymatrix@C=5mm{ &\mathscr{H}^{i+1}(F(K(n)) \ar[d]\ar[r] 			& \mathscr{H}^{i+1}(F(\mathscr{E}(n)^{\oplus s} ))\ar[r]\ar[d] 				&\mathscr{H}^{i+1}(F(\mathscr{E}(n+1)))\ar[d] \\
0\ar[r] 	& p_{2*}(\mathscr{Bla}^{i+1}\otimes p_{1}^{*}K(n)) \ar[r] & p_{2*}(\mathscr{Bla}^{i+1}\otimes p_{1}^{*}\mathscr{E}(n)^{\oplus s}) \ar[r] & p_{2*}(\mathscr{Bla}^{i+1}\otimes p_{1}^{*}\mathscr{E}(n+1))\ar[r] & 0
}$$
and for $n$ high enough depending on $K$ and $\mathscr{E}$, the vertical arrows are isomorphisms by the induction hypothesis; therefore the top sequence is also exact. Hence, moving down to the $\text{i}^{\text{th}}$ cohomology sheaves, for $n$ sufficiently high we also get a surjection
$$\mathscr{H}^{i}(F(\mathscr{E}(n)))^{\oplus s} \to \mathscr{H}^{i}(F(\mathscr{E}(n+1)))\to 0\mathrm{.}$$
Since each $ \mathscr{H}^{i}(F(\mathscr{E}(n)))$ is coherent, this is enough to conclude that the sheaf
$$\bigoplus_{n>n_{0}}  \mathscr{H}^{i}(F(\mathscr{E}(n)))$$
is finitely generated for each $n_{0}$ as a $\Gamma_{*}(X,\Ol_{X})\otimes \Ol_{Y}$-module, where the $\Gamma_{*}(X,\Ol_{X})$-action comes from the action of $\Gamma_{*}(X,\Ol_{X})$ on $\oplus \mathscr{E}(n)$ which gives a corresponding action on $\oplus F(\mathscr{E}(n))$ and hence on $\oplus \mathscr{H}^{i}(F(\mathscr{E}(n)))$. By Lemma \ref{bla_exists} then, the sheaf $\bigoplus_{n>n_{0}}  \mathscr{H}^{i}(F(\mathscr{E}(n)))$ on $Y$ corresponds to a sheaf $\mathscr{Bla}^{i}_{\mathscr{E}}$ on $X\times Y$ such that the functorial map
$$ \mathscr{H}^{i}(F(\mathscr{E}(n)))\to p_{2*}(\mathscr{Bla}_{\mathscr{E}}^{i}\otimes p_{1}^{*}\Ol_{X}(n))$$
is an isomorphisms for $n$ sufficiently high. 

Now consider the functor
\begin{align*}B:\text{Vect}(X) &\to \text{Coh}(X\times Y) \\
\mathscr{E} &\mapsto \mathscr{Bla}^{i}_{\mathscr{E}}
\end{align*}
from the category of locally free sheaves of finite rank on $X$ to the category of coherent sheaves on $X\times Y$. The functor $B$ is additive and right exact. In fact, given two coherent sheaves $\mathscr{E}_{1}$ and $\mathscr{E}_{2}$, 
$$ \bigoplus_{n}\mathscr{H}^{i}(F((\mathscr{E}_{1}+\mathscr{E}_{2})(n)))=\bigoplus_{n}  \mathscr{H}^{i}(F(\mathscr{E}_{1}(n)))  \oplus  \bigoplus_{n}  \mathscr{H}^{i}(F(\mathscr{E}_{2}(n)))$$
hence the functor is additive. Moreover, given a short exact sequence of finite rank locally free sheaves $0\to\mathscr{E}_{1}\to\mathscr{E}_{2}\to\mathscr{E}_{3}\to 0$, we get a triangle $F(\mathscr{E}_{1})\to F(\mathscr{E}_{2})\to F(\mathscr{E}_{3})$ hence for $n\gg 0$ we have (by induction hypothesis)
$$\xymatrix{ &\mathscr{H}^{i+1}(F(\mathscr{E}_{1}(n)) \ar[r]\ar[d] & \mathscr{H}^{i+1}(F(\mathscr{E}_{2}(n))) \ar[r]\ar[d] &\mathscr{H}^{i+1}(F(\mathscr{E}_{3}(n)))\ar[d] \\
0 \ar[r]& p_{2*}(\mathscr{Bla}^{i+1}\otimes p_{1}^{*}\mathscr{E}_{1}(n)) \ar[r] & p_{2*}(\mathscr{Bla}^{i+1}\otimes p_{1}^{*}\mathscr{E}_{2}(n)) \ar[r] & p_{2*}(\mathscr{Bla}^{i+1}\otimes p_{1}^{*}\mathscr{E}_{3}(n))\ar[r] & 0
}$$
and for $n$ sufficiently high, all of the vertical maps are isomorphisms hence the top sequence is also exact for $n$ high, say $n>n_{0}$. 

Hence moving down to the $i^{th}$ cohomology sheaves we get
$$\bigoplus_{n>n_{0}} \mathscr{H}^{i}(F(\mathscr{E}_{1}(n)))\to \bigoplus_{n>n_{0}} \mathscr{H}^{i}(F(\mathscr{E}_{2}(n)))\to  \bigoplus_{n>n_{0}}\mathscr{H}^{i}(F(\mathscr{E}_{3}(n)))\to 0$$
and so (by the equivalence of categories) get
$$\mathscr{Bla}^{i}_{\mathscr{E}_{1}}\to \mathscr{Bla}^{i}_{\mathscr{E}_{2}}\to \mathscr{Bla}^{i}_{\mathscr{E}_{3}}\to 0$$
hence the functor is right exact on the full subcategory of locally free sheaves of finite rank.

Moreover, for every $n$, for $m \gg0$ (depending on $n$) we have
$$\mathscr{H}^{i}(F(\mathscr{E}(n)(m)))= p_{2*}(\mathscr{Bla}^{i}_{\mathscr{E}(n)}\otimes p_{1}^{*}\Ol_{X}(m))$$
but also
\begin{align*} 
\mathscr{H}^{i}(F(\mathscr{E}(n)(m))) &=\mathscr{H}^{i}(F(\mathscr{E}(n+m))) \\
&=p_{2*}(\mathscr{Bla}^{i}_{\mathscr{E}}\otimes p_{1}^{*}\Ol_{X}(n+m)) \\
&=p_{2*}((\mathscr{Bla}^{i}_{\mathscr{E}}\otimes p_{1}^{*}\Ol_{X}(n))\otimes p_{1}^{*}\Ol_{X}(m))\mathrm{,}
\end{align*}
hence it follows from the equivalence of categories that 
$$\mathscr{Bla}^{i}_{\mathscr{E}(n)}=\mathscr{Bla}^{i}_{\mathscr{E}}\otimes \Ol_{X}(n)\mathrm{.}$$

Now let $\mathscr{E}$ be a locally free sheaf of finite rank on $X$. Then there exists a sequence 
$$\oplus \Ol_{X}(b_{j})\to \oplus \Ol_{X}(a_{k})\to \mathscr{E}\to 0\mathrm{,}$$
therefore since the functor $B$ is exact we get
$$\mathscr{Bla}^{i}_{ \oplus \Ol_{X}(b_{j})} \to \mathscr{Bla}^{i}_{ \oplus \Ol_{X}(a_{k})} \to \mathscr{Bla}^{i}_{\mathscr{E}}\to 0\mathrm{.}$$
Since $B$ is additive and compatible with twists, we can write
$$\oplus \mathscr{Bla}^{i}_{\Ol_{X}}\otimes p_{1}^{*}\Ol_{X}(b_{j}) \to \oplus \mathscr{Bla}^{i}_{\Ol_{X}}\otimes p_{1}^{*}\Ol_{X}(a_{k}) \to \mathscr{Bla}^{i}_{\mathscr{E}}\to 0$$
hence
$$\mathscr{Bla}^{i}_{\mathscr{E}}=\mathscr{Bla}^{i}_{\Ol_{X}}\otimes p_{1}^{*}\mathscr{E}$$
and the theorem follows by taking $\mathscr{Bla}^{i}=\mathscr{Bla}^{i}_{\Ol_{X}}$. Since there is a finite number of steps in the induction, we can find an $n_{0}$ such that for $n > n_{0}$ the maps
$$ \mathscr{H}^{i}(F(\mathscr{E}(n)))\to p_{2*}(\mathscr{Bla}_{\mathscr{E}}^{i}\otimes p_{1}^{*}\Ol_{X}(n))$$ 
are isomorphisms for all $i$.

In the case where $F\cong \Phi_{E}$, a Fourier-Mukai transform with kernel $E$, we have the following:
\begin{align*}
 \mathscr{H}^{i}(\Phi_{E}(\mathscr{O}_{X}(n))) &= \mathscr{H}^{i} (Rp_{2*}(E\stackrel{L}{\otimes}Lp_{1}^{*} \Ol_{X}(n)))	\cong \\
&\cong \mathscr{H}^{i} (p_{2*}(E \otimes p_{1}^{*} \Ol_{X}(n))) \cong \\
&\cong p_{2*}(\mathscr{H}^{i} (E \otimes p_{1}^{*} \Ol_{X}(n)))\cong \\
&\cong p_{2*}(\mathscr{H}^{i} (E) \otimes p_{1}^{*} \Ol_{X}(n))
\end{align*}
where for the second equality we used the fact that, for $n\gg 0$, $p_{1}^{*}\Ol_{X}$ is very ample with respect to $X\times Y\to Y$. Therefore, by Lemma \ref{bla_exists} the sheaf
$$\bigoplus_{n\geq n_{0}}\mathscr{H}^{i}(\Phi_{E}(\mathscr{O}_{X}(n)))  \cong \bigoplus_{n\geq n_{0}}p_{2*}(\mathscr{H}^{i} (E) \otimes p_{1}^{*} \Ol_{X}(n))$$
corresponds by the equivalence of categories to the sheaf $\mathscr{H}^{i}(E)$ on $X\times Y$.
\end{proof}


%

%

While Theorem \ref{all blas} gives a map $ \mathscr{H}^{i}(F(\mathscr{E}(n)))\to p_{2*}(\mathscr{Bla}^{i} \otimes p_{1}^{*}\mathscr{E}(n))$ for all vector bundles $\mathscr{E}$ and all $n\in \mathbb{Z}$, in general it is only an isomorphism for $n$ sufficiently large. In the case of the first $M$ such that $\mathscr{H}^{M}(F(\mathscr{E}))$ is nonzero for some locally free sheaf of finite rank $\mathscr{E}$ we can actually say more:

\begin{proposition}\label{bla_for_free}
In the situation of Theorem \ref{all blas}, assume $F(\mathscr{E})\in D^{[M,N]}_{\mathrm{Coh}}(Y)$ for all locally free sheaves $\mathscr{E}$ of finite rank on $X$. Then the maps
$$ \mathscr{H}^{M}(F(\mathscr{E}))\to p_{2*}(\mathscr{Bla}^{M} \otimes p_{1}^{*}\mathscr{E})$$
are isomorphisms for all locally free sheaves $\mathscr{E}$ of finite rank.
\end{proposition}

As we mentioned in the proof of Theorem \ref{all blas}, the assumption that $F(\mathscr{E})\in D^{[M,N]}_{\mathrm{Coh}}(Y)$ for all locally free sheaves $\mathscr{E}$ of finite rank on $X$ isn't actually restrictive because of \cite[Lemma 2.4]{orlov}.

\begin{proof}
Assume we have an immersion $X\hookrightarrow 	\pr_{k}^{d}$. Choose sections $s_{1},\ldots,s_{d+1}$ of $\Ol_{X}(1)$ such that the corresponding hyperplanes have empty intersection. Then for any $m\in \N$ we have short exact sequence
$$0\to \Ol_{X}\xrightarrow{(s_{1}^{m},\ldots, s_{d+1}^{m})} \Ol_{X}(m)^{d+1}\to K_{m}\to 0\mathrm{,}$$
where $K_{m}$ is a locally free sheaf on $X$. 

Let $\mathscr{E}$ be any coherent locally free sheaf. Then by tensoring the above short exact sequence with $\mathscr{E}$ we get
$$0\to \mathscr{E}\to \mathscr{E}(m)^{\oplus (d+1)} \to K_{m}\otimes \mathscr{E}\to 0$$
and so
$$\xymatrix{ 0\ar[r]&\mathscr{H}^{M}(F(\mathscr{E})) \ar[r]\ar[d] & \mathscr{H}^{M}(F(\mathscr{E}(m)^{d+1})) \ar[r]\ar[d] &\mathscr{H}^{M}(F(K_{m}\otimes \mathscr{E}))\ar[d] \\
0 \ar[r]& p_{2*}(\mathscr{Bla}^{M}\otimes p_{1}^{*}\mathscr{E}) \ar[r] & p_{2*}(\mathscr{Bla}^{M}\otimes p_{1}^{*}\mathscr{E}(m)^{d+1}) \ar[r] & p_{2*}(\mathscr{Bla}^{M}\otimes p_{1}^{*}(K_{m}\otimes\mathscr{E}))\mathrm{.}
}$$

Let $m$ be high enough so that the center map is an isomorphism (this is possible by Theorem \ref{all blas}). Then the map on the left must be injective. Thus we showed: for every locally free sheaf $\mathscr{E}$ of finite rank, the map $ \mathscr{H}^{M}(F(\mathscr{E}))\to p_{2*}(\mathscr{Bla}^{M} \otimes p_{1}^{*}\mathscr{E})$ is injective. 

Now let us go back to the diagram above. By what we just showed, the map on the right $ \mathscr{H}^{M}(F(K_{m}\otimes \mathscr{E}))\to p_{2*}(\mathscr{Bla}^{M} \otimes p_{1}^{*}(K_{m}\otimes\mathscr{E}))$ is injective. Hence we have 
$$\xymatrix{ 0\ar[r]&\mathscr{H}^{M}(F(\mathscr{E})) \ar[r]\ar@{^{(}->}[d] & \mathscr{H}^{M}(F(\mathscr{E}(m)^{d+1})) \ar[r]\ar[d]^{\cong} &\mathscr{H}^{M}(F(K_{m}\otimes \mathscr{E}))\ar@{^{(}->}[d] \\
0 \ar[r]& p_{2*}(\mathscr{Bla}^{M}\otimes p_{1}^{*}\mathscr{E}) \ar[r] & p_{2*}(\mathscr{Bla}^{M}\otimes p_{1}^{*}\mathscr{E}(m)^{d+1}) \ar[r] & p_{2*}(\mathscr{Bla}^{M}\otimes p_{1}^{*}(K_{m}\otimes\mathscr{E}))
}$$
then by the 5 Lemma the left arrow is an isomorphism, i.e.
$$ \mathscr{H}^{M}(F(\mathscr{E}))\xrightarrow{\cong} p_{2*}(\mathscr{Bla}^{M} \otimes p_{1}^{*}\mathscr{E})\mathrm{.}$$
\vskip-.5cm
\end{proof}

This proposition in particular implies that if $F(\mathscr{E})\in D^{[M,N]}_{\mathrm{Coh}}(Y)$ for all locally free sheaves $\mathscr{E}$ of finite rank on $X$, and there is at least one $\mathscr{E}$ such that $\mathscr{H}^{M}(F(\mathscr{E}))\neq 0$, then necessarily $\mathscr{Bla}^{M}\neq 0$.

Similarly to Proposition \ref{bla_for_free}, we also have a stronger result than the one in Theorem \ref{all blas} for the largest $N'$ such that  $\mathscr{Bla}^{i}\neq 0$. In this case, the map $\mathscr{H}^{N'}(F(\mathscr{E}(n))\to p_{2*}(\mathscr{Bla}^{N'}\otimes p_{1}^{*}\mathscr{E}(n))$ can be constructed for all coherent sheaves on $X$ instead of just the locally free ones:

\begin{proposition}\label{lastbla}
In the situation of Theorem \ref{all blas}, let $N'$ be the largest $i$ such that $\mathscr{Bla}^{i}\neq 0$. Then for all $n\in \Z$, for any coherent sheaf $\mathscr{F}$ we have a map 
$$\mathscr{H}^{N'}(F(\mathscr{F}(n))\to p_{2*}(\mathscr{Bla}^{N'}\otimes p_{1}^{*}\mathscr{F}(n))$$
which is an isomorphism for $n$ sufficiently high.
\end{proposition}

\begin{proof}
First of all, notice that for any coherent sheaf $\mathscr{F}$ on $X$ we have $\mathscr{H}^{N'+1}(\mathscr{F}(n))=0$ for $n\gg 0$. This is because we can always find a short exact sequence $0\to \mathscr{E}_{1}\to \mathscr{E}_{2} \to\mathscr{F}\to 0$ with $\mathscr{E}_{i}$ locally free of finite rank and so, by Theorem \ref{all blas}, $\mathscr{H}^{N'+1}(F(\mathscr{E}_{i}(n)))=p_{2*}(\mathscr{B}^{N'+1}\otimes p_{1}^{*}(\mathscr{E}_{i}))=0$ for $n\gg 0$ since $\mathscr{B}^{N'+1}=0$.

Now consider a coherent sheaf $\mathscr{F}$ on $X$. As in Theorem \ref{all blas} we have a short exact sequence $0\to K(n)\to \mathscr{F}^{\oplus s}(n) \to \mathscr{F}(n+1)\to 0$ ($K$ now is no longer locally free) hence for $n\gg 0$ we get 
$$\mathscr{H}^{N'}(F(\mathscr{F}(n)))^{\oplus s} \to \mathscr{H}^{N'}(F(\mathscr{F}(n+1)))\to 0$$
since $\mathscr{H}^{N'+1}(F(K(n)))=0$. As in the proof of Theorem \ref{all blas} this gives us a $\mathscr{B}_{\mathscr{F}}^{N'}$ and a map
$$\mathscr{H}^{N'}(F(\mathscr{F}(n)))\to p_{2*}(\mathscr{B}_{\mathscr{F}}^{N'}\otimes p_{1}^{*}(\mathscr{O}_{X}(n)))$$
which is an isomorphism for $n$ big.

Now let us go back to a short exact sequence $0\to \mathscr{E}_{1}\to \mathscr{E}_{2} \to\mathscr{F}\to 0$ with $\mathscr{E}_{i}$ locally free of finite rank. Since $\mathscr{E}_{1}$ is locally free, $\mathscr{H}^{N'+1}(F(\mathscr{E}_{1}(n)))=0$ for $n\gg 0$ and the exact sequence
$$\mathscr{H}^{N'}(F(\mathscr{E}_{1}(n)))\to \mathscr{H}^{N'}(F(\mathscr{E}_{2}(n))) \to \mathscr{H}^{N'}(F(\mathscr{F}(n))) \to 0$$
gives us a an exact sequence as in the proof of Theorem \ref{all blas} 
$$\mathscr{Bla}^{N'}_{\mathscr{E}_{1}}\to \mathscr{Bla}^{N'}_{\mathscr{E}_{2}}\to \mathscr{Bla}^{N'}_{\mathscr{F}}\to 0\mathrm{,}$$
therefore we can conclude that $\mathscr{B}_{\mathscr{F}}^{N'}=\mathscr{B}^{N'}\otimes p_{1}^{*}(\mathscr{F})$ for any coherent sheaf $\mathscr{F}$.
\end{proof}

\section{A special case}

In this section we will give an example of a class of exact functors $F:D^{b}_{\mathrm{Coh}}(X)\to D^{b}_{\mathrm{Coh}}(Y)$ for which we can always find an object $E\in D^{b}_{\mathrm{Coh}}(X\times Y)$ and an equivalence $F\cong\Phi_{E}$. The sheaves $\mathscr{Bla}^{i}$ will be the ones defined as in Theorem \ref{all blas}.

In what follows we will take $\text{dim}(X)$ to be equal to one, all but one of the $\mathscr{Bla}^{i}$'s to be equal to zero, and the nonzero one to be supported at a finite number of points.

Before we proceed to find an isomorphism of functors, we need to show that we can obtain an isomorphism on the cohomology sheaves. The following proposition gives a description of the cohomology sheaves of $F(\mathscr{F})$ for any $\mathscr{F}$ coherent, without needing to twist by some high $n$ as we did in Theorem \ref{all blas}.

\begin{proposition}\label{oncohomology}
Let $X$, $Y$ be smooth projective varieties over an algebraically closed field, with $X$ of dimension one, let $F:D^{b}_{\mathrm{Coh}}(X)\to D^{b}_{\mathrm{Coh}}(Y)$ an exact functor, and assume that the sheaves $\mathscr{Bla}^{i}$ defined as in Theorem \ref{all blas} are zero for $i \neq M$. Assume also that $\mathscr{Bla}^{M}$ is a coherent sheaf supported at finitely many points of $X\times Y$. 

Then for any coherent sheaf $\F$ on $X$ we have $\mathscr{H}^{i}(F(\F))=0$ for $i\neq M, M-1$ and for any locally free sheaf $\mathscr{E}$ of finite rank we have $\mathscr{H}^{i}(F(\mathscr{E}))=0$ for $i\neq M$. 

Moreover, for each coherent sheaf $\F$ on $X$ there is a functorial isomorphism 
$$ \mathscr{H}^{M}(F(\F))\xrightarrow{\cong} p_{2*}(\mathscr{Bla}^{M} \otimes p_{1}^{*}\F)\mathrm{.}$$
\end{proposition}

\begin{proof}
Consider any torsion sheaf $Q$. Then we have a short exact sequence of coherent sheaves $0\to \mathscr{E}' \to \mathscr{E} \to Q \to 0$ with $\mathscr{E}, \mathscr{E}'$ locally free. Twist $\mathscr{E}$ and $\mathscr{E}'$ by $n\gg0$ so that $\mathscr{H}^{i}(F(\mathscr{E}'(n)))=\mathscr{H}^{i}(F(\mathscr{E}(n)))=0$ for $i\neq M$. Since $0\to \mathscr{E}'(n) \to \mathscr{E}(n) \to Q \to 0$ is still an exact sequence, from the long exact sequence on cohomology we can conclude that $\mathscr{H}^{i}(F(Q))=0$ for all $i\neq M, M-1$.

Now consider a locally free sheaf $\mathscr{E}$ of finite rank on $X$. Let $\bar{n}$ be large enough so that we know $\mathscr{H}^{i}(F(\mathscr{E}(\bar{n})))=0$ for all $i\neq M$. Then we have a short exact sequence $0\to \mathscr{E}(\bar{n}-1) \to \mathscr{E}(\bar{n}) \to T \to 0$ where $T$ is a torsion sheaf. A portion of the long exact sequence in cohomology gives
$$\mathscr{H}^{i-1}(F(T))\to \mathscr{H}^{i}(F(\mathscr{E}(\bar{n}-1))) \to \mathscr{H}^{i}(F(\mathscr{E}(\bar{n})))$$
and $\mathscr{H}^{i-1}(F(T))= \mathscr{H}^{i}(F(\mathscr{E}(\bar{n})))=0$ for $i\neq M, M+1$ hence $\mathscr{H}^{i}(F(\mathscr{E}(\bar{n}-1)))=0$ for $i\neq M,M+1$. By descending induction on $n$ we then obtain that $\mathscr{H}^{i}(F(\mathscr{E}(n)))=0$ for all $n$ and $i\neq M, M+1$. We will show at the end of the proof that $\mathscr{H}^{M+1}(F(\mathscr{E}))=0$.

By Proposition \ref{lastbla} we know that for any coherent sheaf $\mathscr{F}$ on $X$ we have a functorial map
$$ \mathscr{H}^{M}(F(\F))\rightarrow p_{2*}(\mathscr{Bla}^{M} \otimes p_{1}^{*}\mathscr{F})$$
which is an isomorphism by Proposition \ref{bla_for_free} if $\mathscr{F}$ is locally free (notice that the hypotheses of \ref{bla_for_free} are satisfied by the first part of this Proposition). Moreover we also know, again by Proposition \ref{lastbla}, that for any coherent sheaf $\mathscr{F}$ the map
$$ \mathscr{H}^{M}(F(\mathscr{F}(n)))\xrightarrow{\cong} p_{2*}(\mathscr{Bla}^{M} \otimes p_{1}^{*}\mathscr{F}(n))$$
is an isomorphism for $n$ sufficiently high. But when $\mathscr{F}$ is a sheaf supported at a point twisting doesn't affect the sheaf, so we get that
$$ \mathscr{H}^{M}(F(\mathscr{F}))\xrightarrow{\cong} p_{2*}(\mathscr{Bla}^{M} \otimes p_{1}^{*}\mathscr{F})$$
is also an isomorphism for torsion sheaves, and hence it is always an isomorphism since any coherent sheaf on $X$ is the direct sum of a locally free part and a torsion part. 

Now let us show that $\mathscr{H}^{M+1}(F(\mathscr{E}))=0$ for a locally free sheaf $\mathscr{E}$:  using again the short exact sequence $0\to \mathscr{E}(\bar{n}-1) \to \mathscr{E}(\bar{n}) \to T \to 0$, we obtain a diagram
$$\xymatrix@C=4mm{
 \mathscr{H}^{M}(F(\mathscr{E}(\bar{n}))) \ar[r]\ar[d]^{\cong} &  \mathscr{H}^{M}(F(T)) \ar[r]\ar[d]^{\cong} &  \mathscr{H}^{M+1}(F(\mathscr{E}(\bar{n}-1))) \ar[r]\ar[d] & \mathscr{H}^{M+1}(F(\mathscr{E}(\bar{n})))=0 \\
 p_{2*}(\mathscr{Bla}^{M} \otimes p_{1}^{*}\mathscr{E}(\bar{n})) \ar[r] &  p_{2*}(\mathscr{Bla}^{M} \otimes p_{1}^{*}T) \ar[r] & 0
}$$
where the bottom sequence is right exact because $p_{2*}$ is exact when applied to a sequence of flasque sheaves, which is the case here by the way we chose our 
$\mathscr{Bla}^{M}$. From the five lemma it follows that $\mathscr{H}^{M+1}(F(\mathscr{E}(\bar{n}-1)))=0$. So we can again proceed by descending induction on $n$.
\end{proof}


Thanks to Proposition \ref{oncohomology} we can now get an isomorphism of the $\delta$-functors obtained by first applying the two functors $F$ and $\Phi_{\mathscr{B}^{M}[-M]}$ and then taking their cohomology:

\begin{proposition}\label{isoforfree}
In the setting of Proposition \ref{oncohomology}, let $\Phi$ be the Fourier-Mukai transform associated to the complex given by the sheaf $\mathscr{Bla}^{M}$ placed in degree $M$, $\Phi=\Phi_{\mathscr{Bla}_{M}[-M]}$.

Then there is an isomorphism of $\delta$-functors
$$ \mathscr{H}^{i}(F(\cdot))\xrightarrow{\cong} \mathscr{H}^{i}(\Phi(\cdot))$$
on the  category of coherent sheaves on $X$, which gives an isomorphism of functors $F\to \Phi$ for the full subcategory of $D^{b}_{\mathrm{Coh}}(X)$ consisting of locally free sheaves placed in degree zero.
\end{proposition}

\begin{proof}
The fact that there is a functorial isomorphism
$$ \mathscr{H}^{M}(F(\cdot))\xrightarrow{\cong} \mathscr{H}^{M}(\Phi(\cdot))$$
on the category of coherent sheaves on $X$ follows immediately from Proposition \ref{oncohomology} given that $\mathscr{H}^{M}(\Phi(\F))=p_{2*}(\mathscr{Bla}^{M} \otimes p_{1}^{*}\F)$ since pushforward is exact for flasque sheaves,. 

Moreover, for any locally free sheaf of finite rank $\mathscr{E}$, since the only nonzero cohomology sheaf of $F(\mathscr{E})$ is in degree $M$, 
\begin{align*}
F(\mathscr{E}) &= \mathscr{H}^{M}(F(\mathscr{E}))[-M]\xrightarrow{\cong} \mathscr{H}^{M}(\Phi(\mathscr{E}))[-M]= \\
&=p_{2*}(\mathscr{Bla}^{M} \otimes p_{1}^{*}\mathscr{E})[-M]=Rp_{2*}(\mathscr{Bla}^{M} [-M]\stackrel{L}{\otimes} Lp_{1}^{*}\mathscr{E})=\Phi(\mathscr{E})\mathrm{,}
\end{align*}
where the third equality follows again by Proposition \ref{oncohomology}
This gives the isomorphism of functors on the full subcategory of $D^{b}_{\mathrm{Coh}}(X)$ of locally free sheaves placed in degree zero.

Let us now construct the isomorphism
$$ \mathscr{H}^{M-1}(F(\cdot))\xrightarrow{\cong} \mathscr{H}^{M-1}(\Phi(\cdot))\mathrm{.}$$
Consider a coherent sheaf $Q$ on $X$ which is not locally free. Then there is a short exact sequence of coherent sheaves $0\to A'\to A \to Q\to 0$ where $A'$ and $A$ are locally free. We get a long exact sequence on cohomology
$$\xymatrix{
0\ar[r] &\mathscr{H}^{M-1}(F(Q))\ar[r]      &\mathscr{H}^{M}(F(A')) \ar[r]\ar[d] & \mathscr{H}^{M}(F(A)) \ar[d]  \\
0 \ar[r] &\mathscr{H}^{M-1}(\Phi(Q))\ar[r] & \mathscr{H}^{M}(\Phi(A')) \ar[r] & \mathscr{H}^{M}(\Phi(A)) 
}$$
so we get an isomorphism $\mathscr{H}^{M-1}(F(Q))\to \mathscr{H}^{M-1}(\Phi(Q))$. 

We still need to show that this map is functorial and that it does not depend on the choice of a short exact sequence. Consider a map $Q\to T$ of coherent sheaves. Then we can construct two short exact sequences
$$\xymatrix{
0\ar[r] & A'\ar[r]\ar[d] & A \ar[r]\ar[d] & Q\ar[r]\ar[d]&  0 \\
0\ar[r] & B' \ar[r] & B\ar[r] & T \ar[r] & 0
}$$
with $A, A', B$ and $B'$ locally free of finite rank. Then we get the following diagram on cohomology:
\psset{unit=1.5,dimen=middle}
\begin{pspicture}(0,3)
\pspolygon[dimen=middle,linestyle=none,fillstyle=solid,fillcolor= lightgray](1.9,0)(3.35,.85)(3.35,2.6)(1.9,1.7)
\pspolygon[dimen=middle,linestyle=none,fillstyle=solid,fillcolor= lightgray](5.05,0)(6.35,.85)(6.35,2.6)(5.05,1.7)
\pspolygon[dimen=middle,linestyle=none,fillstyle=solid,fillcolor= lightgray](7.95,0)(9.15,.85)(9.15,2.57)(7.95,1.74)
 \end{pspicture}

\vskip-5cm

\begin{equation}\label{delta functor}
\xymatrix@C=.1mm{
&0\ar[rr] &	& \mathscr{H}^{M-1}(\Phi(Q))\ar[rr]\ar '[d][dd] &		&\mathscr{H}^{M}(\Phi(A'))\ar[rr]\ar '[d][dd] &			& \mathscr{H}^{M}(\Phi(A))\ar[dd] \\
0\ar[rr] &	& \mathscr{H}^{M-1}(F(Q)) \ar[ur] \ar[rr]\ar[dd] &		& \mathscr{H}^{M}(\F(A'))\ar[ur]\ar[rr]\ar[dd] &		& \mathscr{H}^{M}(F(A))\ar[ur]\ar[dd]\\
&0\ar'[r][rr]&	& \mathscr{H}^{M-1}(\Phi(T))\ar'[r][rr] &		&\mathscr{H}^{M}(\Phi(B'))\ar'[r][rr] &			& \mathscr{H}^{M}(\Phi(B)) \\
0\ar[rr] & 	&\mathscr{H}^{M-1}(F(T)) \ar[ur] \ar[rr]&		& \mathscr{H}^{M}(F(B'))\ar[ur]\ar[rr] &		& \mathscr{H}^{M}(F(B))\ar[ur]
}\end{equation}
and since the two rightmost shaded squares commute, the leftmost shaded square will also commute. This shows functoriality.  

To show that the maps we chose do not depend on the choice of a short exact sequence, notice that given two short exact sequences $0\to A' \to A \to Q \to 0$ and $0\to B'\to B\to Q\to 0$ there is a short exact sequence $0\to C\to A\oplus B \to Q \to 0$ mapping to both of them. So we just need to prove this statement for two short exact sequences with maps between them. But then we are again in the situation of diagram (\ref{delta functor}), where $T=Q$ and the two rightmost maps in the diagram are the identity. So this follows again from the commutativity of the leftmost diagonal square.

Finally, we have to show that for every short exact sequence $0\to B' \to B \to Q \to 0$ the diagram
$$\xymatrix{
\mathscr{H}^{M-1}(F(Q))\ar[r] \ar[d]      &\mathscr{H}^{M}(F(B'))\ar[d] \\
\mathscr{H}^{M-1}(\Phi(Q))\ar[r] & \mathscr{H}^{M}(\Phi(B'))
}$$
is commutative. This follows immediately by the construction when $B'$ and $B$ are locally free. Otherwise, construct a diagram
$$\xymatrix{
0\ar[r] & A'\ar[r]\ar[d] & A \ar[r]\ar[d] & Q\ar[r]\ar@{=}[d]&  0 \\
0\ar[r] & B' \ar[r] & B\ar[r] & Q \ar[r] & 0
}$$
with $A, A'$ locally free of finite rank. Then we get a diagram as in (\ref{delta functor}) with $T=Q$ and where everything commutes except possibly for the bottom leftmost square, but that follows immediately since the leftmost vertical arrow is the identity.
\end{proof}

We are now ready to tackle the task of getting an isomorphism between the two functors in our special case. We will now assume that $\mathscr{B}^{M}$ is a direct sum of skyscraper sheaves. We will proceed as follows: we will first find an isomorphism for the subcategory of $D^{b}_{\mathrm{Coh}}(X)$ given by sheaves placed in degree zero. The isomorphism on the whole derived category will then follow by the technical Lemma \ref{generality}.

\begin{theorem}\label{onebla}
Let $X$ and $Y$ be smooth projective varieties over an algebraically closed field, with $X$ of dimension one, $F:D^{b}_{\mathrm{Coh}}(X)\to D^{b}_{\mathrm{Coh}}(Y)$ an exact functor. Assume that the corresponding $\mathscr{Bla}^{i}$ defined in Theorem \ref{all blas} are zero for $i \neq M$, and that $\mathscr{Bla}^{M}$ is a skyscraper sheaf supported at a finite number of points, $\mathscr{Bla}^{M}=\bigoplus_{j=1}^{t} k(p_{i},q_{i})$. Let $\Phi$ be the Fourier-Mukai transform associated to the sheaf $\mathscr{Bla}^{M}$ placed in degree $M$. Restrict the two functors to the full subcategory of coherent sheaves placed in degree 0. Then there exists an isomorphism of triangulated functors $s(\cdot):  \Phi(\cdot)\to F(\cdot) $.
\end{theorem}

Before we prove the theorem, let us prove two technical lemmas that we will use in the proof. 

\begin{lemma}\label{mapsanddiagrams}
Let $X$ be a projective variety, $\Ol_{X}(1)$ be a very ample invertible sheaf on $X$. Consider a surjective map $\alpha:\oplus_{n}\Ol_{X}\to Q$ where $Q$ is a coherent sheaf on $X$. Then there exists an integer $h(\alpha)$ such that for all $m\geq h(\alpha)$ and for any map $\beta:\Ol_{X}(-m)\to Q$ there exists a map $\gamma:\Ol_{X}(-m)\to \oplus_{n}\Ol_{X}$ making the following diagram commute:
$$\xymatrix{
\Ol_{X}(-m) \ar@{-->}[rr]^{\gamma}\ar[dr]_{\beta} && \bigoplus_{t} \Ol_{X}\ar[dl]^{\alpha} \\
& Q
}$$
\end{lemma}

\begin{proof}
We have a short exact sequence
$$0\to \mathrm{Ker}(\alpha)\to \oplus_{n} \Ol_{X}\to Q\to 0\mathrm{.}$$
Twist by $\Ol_{X}(m)$ to get 
$$0\to \mathrm{Ker}(\alpha)(m)\to \oplus_{n} \Ol_{X}(m)\to Q(m)\to 0\mathrm{.}$$
A map $\beta:\Ol_{X}(-m)\to Q$ is the same thing as a map $\Ol_{X}\to Q(m)$, hence as an element $\beta(m)\in H^{0}(X,Q(m))$. By Serre vanishing, there exists an $h(\alpha)\geq 0$ such that $H^{1}(X,\mathrm{Ker}(\alpha))(m)=0$ for all $m\geq h(\alpha)$. Hence $\beta(m)$ lifts to a section $\gamma(m)$ of $H^{0}(X, \oplus_{n}\Ol_{X}(m))$. Twist down by $m$ to get the desired map $\gamma: \Ol_{X}(-m)\to \oplus_{n} \Ol_{X}$.
\end{proof}


\begin{lemma}\label{indep sections}
Let $X$ be a smooth projective variety over an algebraically closed field, let $p_{1}, \ldots, p_{t}\in X$ and let $\mathscr{E}$ be a locally free sheaf of rank $r$ generated by global sections. Then there exist an open set $\mathscr{U}$ containing $p_{1},\ldots,p_{t}$ and global sections $s_{1}, \ldots, s_{r}$ of $\mathscr{E}$ that generate the stalk $\mathscr{E}_{p}$ at each point $p\in \mathscr{U}$.
\end{lemma}

\begin{proof}
Assume we found $s_{1}, \ldots, s_{n}\in \Gamma(X,\mathscr{E})$ that are linearly independent at each stalk at $p_{1},\ldots,p_{t}$ so that we have 
\begin{align*}
0\to \Ol_{X}^{\oplus n} &\to \mathscr{E} \xrightarrow{f} Q\to 0 \\
e_{j} &\mapsto s_{j}\mathrm{.}
\end{align*}
Let us find a global section of $\mathscr{E}$ such that its image in $Q$ doesn't vanish at $p_{1}, \ldots, p_{t}$. Let $u_{i}\in \Gamma(X,\mathscr{E})$ such that $f(u_{i})$ doesn't vanish at $p_{i}$ (we can do this because $f$ is surjective on stalks and $\mathscr{E}$ is generated by global sections). Then $u_{1},\ldots, u_{t}$ form a sub-vector space $V$ of $\Gamma(X,\mathscr{E})$ of dimension $l$ for some $l$ and, for each $i$, $\dim(\{u\in V : f(u)(p_{i})=0\})\leq l-1$. Hence
$$\{u\in V : f(u)(p_{i})=0 \text{ for some i}\} =\bigcup_{i} \{u\in V : f(u)(p_{i})=0\}$$
is a union of subsets of dimension less or equal to $l-1$ and hence it is strictly contained in $V$ since our field of definition is infinite (because it is algebraically closed). So we can find a section $s_{n+1}$ in $V$ such that $f(s_{n+1})$ doesn't vanish at any of the $p_{j}$. Then $s_{1}, \ldots, s_{n+1}$ are linearly independent at each $p_{j}$ as sections of $\mathscr{E}$. We can keep doing this as long as $\text{rk}Q>0$. Then the sections $s_{1}, \ldots, s_{r}$ will generate the stalk $\mathscr{E}_{p}$ at each point $p$ in an open set $\mathscr{U}$ containing $p_{1}, \ldots, p_{t}$.
\end{proof}

\begin{proof}[Proof of Theorem \ref{onebla}]
We will first construct the isomorphism on objects, starting with the subcategories of coherent sheaves on $X$ given by locally free sheaves and torsion sheaves. This will a priori involve making non-canonical choices, but as it later turns out, the choices we are making are actually unique. Then we will prove that the isomorphisms are compatible with morphisms and this will allow us to define said isomorphism on a general coherent sheaf. Lastly, we will show that the given isomorphisms induce maps of triangles when applied to a short exact sequence of sheaves.

\textbf{I. On the subcategory of locally free sheaves:} Let $\mathscr{E}$ be a locally free sheaf of finite rank on $X$. Then by Proposition \ref{isoforfree} there is a functorial isomorphism $s(\mathscr{E}): \Phi(\mathscr{E})\to F(\mathscr{E})$. 

\textbf{II. On torsion sheaves: } Consider a torsion sheaf $Q$ on $X$. There exists a short exact sequence $0\to K\to \Ol_{X}^{\oplus n}  \xrightarrow{\alpha}  Q\to 0$, with $K$ a locally free sheaf. Then we have a diagram
$$\xymatrix{
\Phi(K) \ar[r] \ar[d]^{s(K)}_{qis}		& \Phi(\Ol_{X}^{\oplus n}) \ar[r] \ar[d]^{s(\Ol_{X}^{\oplus n})}_{qis}		&\Phi(Q)\ar@{-->}[d] \\
F(K) \ar[r]							&F(\Ol_{X}^{\oplus n}) \ar[r]									&F(Q)
}$$
hence there exists a dotted arrow $\Phi(Q)\to F(Q)$ which is a quasi-isomorphism (this dotted arrow is not necessarily unique). Choose one such arrow and call it $s(Q)$. Notice that $s(Q)$ will induce on cohomology the maps that we found in Proposition \ref{isoforfree} because the maps induced on the $M^{th}$ cohomology are the same as the ones in Proposition \ref{isoforfree}, and the maps $\mathscr{H}^{M-1}(\Phi(Q)) \to \mathscr{H}^{M}(\Phi(K))$ and $\mathscr{H}^{M-1}(F(Q)) \to \mathscr{H}^{M}(F(K))$ are injective.

\textbf{III. $s(-)$ is compatible with maps $\beta:\mathscr{E}\to Q$, $\mathscr{E}$ locally free, $Q$ torsion:} First of all we will prove the following: for any map $\beta:\Ol_{X}(i) \to Q$, the diagram 
$$\xymatrix{
\Phi(\Ol_{X}(i)) \ar[d]_{s(\Ol_{X}(i))} \ar[r]^-{\Phi(\beta)}  	&\Phi(Q)\ar[d]^{s(Q)} \\
F(\Ol_{X}(i)) \ar[r]^-{F(\beta)}				& F(Q)
}$$
commutes. Consider first the case where $i\leq -h(\alpha)$ where $h(\alpha)$ is defined as in Lemma \ref{mapsanddiagrams}, and $\alpha$ is as in \textbf{II}. For every map $\beta: \Ol_{X}(i) \to Q$ with $i\leq -h(\alpha)$ we have a diagram 
$$\xymatrix{
\Ol_{X}(i) \ar@{-->}[rr]\ar[dr]_{\beta} && \bigoplus_{n} \Ol_{X}\ar[dl]^{\alpha} \\
& Q &
}$$
By applying the functors $F$ and $\Phi$ we obtain the following diagram:
$$\xymatrix@R=5mm{
									&		&\Phi(\Ol_{X}(i)) \ar[r]^-{\Phi(\beta)} \ar[ddll] \ar'[d][ddd]		&\Phi(Q) \ar[ddll]^(.65){id} \ar[ddd]^{s(Q)}	\\ 
									&			&			\\
\Phi(\Ol_{X}^{\oplus n}) \ar[r]_-{\Phi(\alpha)} \ar[ddd]	&\Phi(Q) \ar[ddd]_{s(Q)}	&									&						\\
									&				&F(\Ol_{X}(i)) \ar[r]^-{F(\beta)} \ar'[dl][ddll]				&F(Q)\ar[ddll]^(.65){id}			\\
									&			&		\\
F(\Ol_{X}^{\oplus n}) \ar[r]_-{F(\delta)} 			&F(Q)
}$$
and the front square commutes, hence the back square will also commute.

Now let $i>-h(\alpha)$ and consider $\beta: \Ol_{X}(i)\to Q$. Then pick any map $\gamma: \Ol_{X}(-h(\alpha))\to \Ol_{X}(i)$ such that $\gamma$ is an isomorphism on an open set containing $p_{1},\ldots, p_{t}$. Then the map $\Phi(\gamma): \Phi(\Ol_{X}(-h(\alpha))) \to \Phi(\Ol_{X}(i))$ is an isomorphism: in fact the map $p_{1}^{*}(\gamma):p_{1}^{*}(\Ol_{X}(-h(\alpha)))\to p_{1}^{*}(\Ol_{X}(i))$ is an isomorphism on an open set containing $(p_{1},q_{1}), \ldots, (p_{t},q_{t})$ and hence we will get an isomorphism when tensoring with a sheaf supported at $(p_{1},q_{1}), \ldots, (p_{t},q_{t})$. By letting $\delta=\beta\circ\gamma$ once again we get a diagram
$$\xymatrix@R=5mm{
											&	   							&\Phi(\Ol_{X}(i)) \ar[r]^-{\Phi(\beta)} \ar'[d][ddd]		&\Phi(Q)  \ar[ddd]	\\ 
									&   &		\\
\Phi(\Ol_{X}(-h(\alpha))) \ar[r]_-{\Phi(\delta)} \ar[ddd] \ar[uurr]^{\Phi(\gamma)} 	&\Phi(Q) \ar[ddd]\ar[uurr]^(.35){id}	&					&		\\
															&								&F(\Ol_{X}(i)) \ar[r]^{F(\beta)}		&F(Q)		\\
															& & \\
F(\Ol_{X}(-h(\alpha))) \ar[r]_-{F(\delta)} \ar'[ur][uurr]									&F(Q)\ar[uurr]_(.35){id}
}$$
since $\Phi(\gamma)$ is a quasi-isomorphism and the front square is commutative, the back square will also commute.

Now consider any map $\beta: \mathscr{E}\to Q$ with $\mathscr{E}$ locally free and $Q$ torsion. Let $m$ be such that $\mathscr{E}(m)$ is generated by global sections, and let $r=\text{rk} \mathscr{E}$. By Lemma \ref{indep sections} we can find $s_{1},\ldots, s_{r}$ global sections of  $\mathscr{E}(m)$ that are linearly independent at each stalk of an open set $\mathscr{U}$ containing $p_{1}, \ldots, p_{t}$. Then the corresponding map $\bigoplus_{r} \Ol_{X} \to \mathscr{E}(m)$ is injective and it is an isomorphism on $\mathscr{U}$. Twisting down by $m$ we get a map $\gamma:\bigoplus_{r} \Ol_{X}(-m) \to \mathscr{E}$ which is an isomorphism on $\mathscr{U}$. By letting $\delta=\beta \circ \gamma$ we get again a diagram like the above one,
$$\xymatrix@R=5mm{
									&						&\Phi(\mathscr{E}) \ar[r]^-{\Phi(\beta)}  \ar'[d][ddd]\ar@{<-}[ddll]+<2.8cm,3mm>_{\Phi(\gamma)} 		&\Phi(Q)  \ar[ddd]^{s(Q)}	\\ 
									& & &\\
\Phi(\bigoplus_{r}\Ol_{X}(-m))\cong  \bigoplus_{r}\Phi(\Ol_{X}(-m)) \ar[r]_-{\Phi(\delta)} \ar@<1.7cm>[ddd]       &\Phi(Q) \ar[ddd]_{s(Q)} \ar[uurr]^(.35){id}&				&				\\
							&						&F(\mathscr{E}) \ar[r]^-{F(\beta)}\ar@{<-}'[dl][ddll]+<2.8cm,3mm>		&F(Q)		\\
							& & & \\
F(\bigoplus_{r}\Ol_{X}(-m)) \cong  \bigoplus_{r}F(\Ol_{X}(-m)) \ar[r]_-{F(\delta)} 					&F(Q)\ar[uurr]^(.35){id}
}$$
and since $\Phi(\gamma)$ is a quasi-isomorphism and the front square commutes, the back square will also commute.

\textbf{IV. $s(-)$ is compatible with maps $\eta: Q\to T$, $Q$ and $T$ torsion:} We need to show that for any map between torsion sheaves $Q\to T$, the corresponding diagram
$$\xymatrix{
\Phi(Q) \ar[r]^{\Phi(\eta)} \ar[d]_{s(Q)} & \Phi(T) \ar[d]^{s(T)} \\
F(Q) \ar[r]^{F(\eta)} 		& F(T) 	
}$$
is commutative. This will also prove that our choice of $s(Q)$ is canonical. To do this, consider a locally free sheaf $\mathscr{E}=\bigoplus_{r} \Ol_{X}$ with a surjection $f:\mathscr{E}\to Q$. For consistency we will represent this situation with a square diagram as before
$$\xymatrix{
\mathscr{E} \ar[r]^{\eta\circ f} \ar[d]_{f}  & T\ar[d]^{id} \\
Q\ar[r]^{\eta} & T
}$$
Then we get the following diagram:
$$\xymatrix@R=5mm{
							&						  &\Phi(Q) \ar[r]^{\Phi(\eta)} 	 \ar'[d][ddd]		&\Phi(T)  \ar[ddd]	\\ 
									&&&  \\
\Phi(\mathscr{E}) \ar[r]_{\Phi(\eta\circ f)} \ar[ddd] \ar[uurr]^{\Phi(f)} 	&\Phi(T) \ar[ddd]\ar[uurr]^(.35){id}  &								&						\\
									&						&F(Q) \ar[r]^{F(\eta)} 		&F(T)		\\
									&&& \\
F(\mathscr{E}) \ar[r]_{F(\eta\circ f)} \ar'[ur][uurr]					&F(T)\ar[uurr]^(.35){id}
}$$
where the front square commutes by III. Hence the back square will also commute after pre-composing with the map $\Phi(\mathscr{E})\to\Phi(Q)$. But then we can conclude that the top right square also commutes: in fact it commutes on cohomology because of Proposition \ref{isoforfree}, so we can apply Lemma \ref{stillzero} below.

\textbf{V. On a general coherent sheaf on $X$:} Let $\F$ be any coherent sheaf on $X$. Then we have a decomposition $\F\cong \F_{T}\oplus \F_{F}$ where $\F_{T}$ is the canonical summand consisting of the torsion part of $\F$ and $\F_{F}$ corresponds to the torsion-free part (this summand is not canonical). Then define $s(\F)=s(\F_{T})\oplus s(\F_{F})$. We need to show that this map doesn't depend on the choice of the decomposition. So consider two such decompositions $\F\cong \F_{T}\oplus \F_{F}$ and $\F\cong \F_{T}\oplus \F'_{F}$ and call $s(\F)$ and $s'(\F)$ respectively the two induced maps on $\Phi(\F)$. Then the identity $\mathscr{F} \to \mathscr{F}$ induces a map  $\alpha: \F_{F}\to \F'_{F} \oplus \F_{T}$, and by I. and III. the following diagram is commutative:
$$\xymatrix{
\Phi(\F_{F}) \ar[r] \ar[d] & \Phi(\F'_{F})\oplus \Phi(\F_{T}) \ar@<-6mm>[d]_{s(\F_{T})} \ar@<+6mm>[d]^{s(\F_{F})}  \\
F(\F_{F}) \ar[r] & F(\F'_{F})\oplus F(\F_{T})
}$$
whereas the diagram for the torsion part is clearly commutative because the induced maps are just the identity. Hence every square in the following diagram is commutative:
$$\xymatrix{
\Phi(\F) \ar[r]^-{\cong}\ar[d]_{s(\F)}\ar@/^8mm/[rrr]^{id} & \Phi(\F_{T}) \oplus  \Phi(\F_{F})\ar@<-6mm>[d]_{s(\F_{T})} \ar@<+6mm>[d]^{s(\F_{F})} \ar[r]^-{id\oplus \Phi(\alpha)} & \Phi(\F_{T}) \oplus  \Phi(\F'_{F})\ar@<-6mm>[d]_{s(\F_{T})} \ar@<+6mm>[d]^{s(\F'_{F})}  \ar[r]^-{\cong} & \Phi(\F)\ar[d]^{s'(\F)}\\
F(\F)\ar[r]^-{\cong}\ar@/_8mm/[rrr]_{id} &F(\F_{T})  \oplus  F(\F_{F})\ar[r]^-{id\oplus F(\alpha)} &F(\F_{T})  \oplus  F(\F'_{F})\ar[r]^-{\cong} & F(\F)\mathrm{.}
}$$
It follows that the external rectangle commutes, which proves precisely that $s(\F)=s'(\F)$.

\textbf{VI. $s(-)$ is compatible with any maps $A\to B$, for $A$ and $B$ coherent sheaves:} Given a map $f: A\to B$, write $A=A_{F}\oplus A_{T}$ and $B=B_{F}\oplus B_{T}$. Then $s$ will be compatible with $\Phi(f)$ and $F(f)$ because it is compatible with the maps $A_{F}\to B_{F}$, $A_{F}\to B_{T}$, and $A_{T}\to B_{T}$.

\textbf{VII. $s(-)$ is compatible with triangles of the type  $0\to A\to B\to C\to 0$ for $A$ and $B$ locally free:} The last thing to show is that given a short exact sequence of coherent sheaves on X, $0\to A\to B\to C\to 0$, the maps $s(A)$, $s(B)$ and $s(C)$ give a morphism of triangles
\begin{equation}\label{triangle}\xymatrix{
\Phi(A)\ar[r]\ar[d]^{s(A)} &\Phi(B)\ar[r]\ar[d]^{s(B)} &\Phi(C)\ar[r]\ar[d]^{s(C)} &\Phi(A)[1] \ar[d]^{s(A)[1]}	\\
F(A) \ar[r] & F(B)\ar[r] & F(C)\ar[r] & F(A)[1]\mathrm{,}
}\end{equation}
i.e. we need to prove that the rightmost square is commutative. First of all we will analyze the map $\Phi(B)\to \Phi(C)$. We know that $\Phi(B)$ is supported in degree $M$, whereas $\Phi(C)$ is supported in degrees $M$ and $M-1$: hence, by \cite{dold}, as a complex we have $\Phi(C)\cong \mathscr{H}^{M}(\phi(C))[-M]\oplus \mathscr{H}^{M-1}(\phi(C))[-M+1]$ (in a non-canonical way). The situation looks as follows:
$$\xymatrix@R=.1mm{
\Phi(B) \ar[r] \ar[ddr]  & \mathscr{H}^{M}(\Phi(C))[-M]  \ar[ddr]  \\ 
& \oplus \\
									& \mathscr{H}^{M-1}(\Phi(C))[-M+1] \ar[r] & \Phi(A)[1]\mathrm{.}	\\  
}$$
We will now show that the induced maps $\Phi(B)\to \mathscr{H}^{M-1}(\Phi(C))[-M+1]$, as well as $\mathscr{H}^{M}(\Phi(C))[-M]\to \Phi(A)[1]$, are zero in $D^{b}_{\mathrm{Coh}}(Y)$ for some choice of a decomposition $\Phi(C)\cong \mathscr{H}^{M-1}(\Phi(C))[-M+1] \oplus \mathscr{H}^{M}(\Phi(C))[-M]$. In fact, consider a locally free resolution of $p_{1}^{*}C$, $\bar{C}_{-1}\to \bar{C}_{0}$. Then the map $B\to C$ induces an actual map of complexes
$$\xymatrix{
(p_{1}^{*}B\otimes \mathscr{Bla}^{M})[-M] \ar[r] & (\bar{C}_{0}\otimes \mathscr{Bla}^{M})[-M] \\
& (\bar{C}_{-1}\otimes \mathscr{Bla}^{M})[-M+1] \ar[u]\mathrm{.}
}$$
Now, since these complexes are direct sums of complexes of vector spaces over $k(p_{i},q_{i})$, we can write the complex on the right as a direct sum of its cohomology groups and get a map of complexes 
$$\xymatrix@R=.1mm{
(p_{1}^{*}B\otimes \mathscr{Bla}^{M})[-M] \ar[r] & (\mathscr{H}^{M}(p_{1}^{*}(C)\stackrel{L}{\otimes}\mathscr{Bla}^{M})) [-M]\\
&\oplus \\
&(\mathscr{H}^{M-1}(p_{1}^{*}(C)\stackrel{L}{\otimes}\mathscr{Bla}^{M})) [-M+1]\mathrm{,}
}$$
and by pushing forward to $Y$ we get a map of complexes
$$\xymatrix@R=.1mm{
\Phi(B) \ar[r] & p_{2*}(\mathscr{H}^{M}(p_{1}^{*}(C)\stackrel{L}{\otimes}\mathscr{Bla}^{M}))[-M]\cong \mathscr{H}^{M}(\Phi(C))[-M] \\
&\oplus \\
& p_{2*}(\mathscr{H}^{M-1}(p_{1}^{*}(C)\stackrel{L}{\otimes}\mathscr{Bla}^{M})) [-M+1] \cong \mathscr{H}^{M-1}(\Phi(C))[-M+1] \mathrm{.}
}$$
This proves precisely that the map $\Phi(B)\to \mathscr{H}^{M-1}(\phi(C))[-M+1]$ is zero ($p_{2*}$ is exact here because the sheaves are flasque). For the second map we can reason as follows: since the map $\Phi(B)\to \mathscr{H}^{M-1}(\phi(C))[-M+1]$ is zero, and we know that $\Phi(B)\to \Phi(A)[1]$ is zero, it follows that the composition $\Phi(B)\to\mathscr{H}^{M}(\Phi(C))[-M] \to \Phi(A)[1]$ is also zero. Hence the result follows if the map 
$$\text{Hom}(\mathscr{H}^{M}(\Phi(C))[-M], \Phi(A)[1]) \to \text{Hom}(\Phi(B), \Phi(A)[1])$$
is injective, i.e. the map 
$$\text{Ext}^{1}(\mathscr{H}^{M}(\Phi(C)), \mathscr{H}^{M}(\Phi(A))) \to \text{Ext}^{1}( \mathscr{H}^{M}(\Phi(B)),  \mathscr{H}^{M}(\Phi(A)))$$
is injective. Assuming that $B$ is locally free of rank $r$ and 
$$C\cong \bigoplus_{i=1}^{t}\bigoplus_{j_{i}} \Ol_{p_{i}}/\mathfrak{m}_{p_{i}}^{h_{j_{i}}}\Ol_{p_{i}}\oplus \text{torsion part supported away from the }p_{i} \oplus \text{ free part of rank }r_{1}\leq r\mathrm{,}$$
a short computation shows that the map in question is
\begin{align*}
&\bigoplus_{i=1}^{t} \bigoplus_{r}  \mathscr{H}^{M}(\Phi(A))/\mathfrak{m}_{q_{i}} \mathscr{H}^{M}(\Phi(A)) \xrightarrow{\alpha} \bigoplus_{i=1}^{t} \bigoplus_{j_{i}} \mathscr{H}^{M}(\Phi(A))/\mathfrak{m}_{q_{i}}  \mathscr{H}^{M}(\Phi(A)) \oplus \bigoplus_{r_{1}} \mathscr{H}^{M}(\Phi(A))/\mathfrak{m}_{q_{i}}  \mathscr{H}^{M}(\Phi(A))\mathrm{,}  \\
& \text{ where }j_{i}\leq r \text{ and there exists a basis such that } \alpha=\left(\begin{array}{ccc} 1 \\ & \ddots \\ & & 1 \\ 0 & \ldots &0\end{array}\right)\mathrm{;}
\end{align*}
hence it is injective as desired.

We're finally ready to show that 
$$\xymatrix{
\Phi(C) \ar[r] \ar[d] & \Phi(A)[1] \ar[d]\\
F(C) \ar[r] & F(A)[1]
}$$
commutes. To do this, take the same decomposition $\Phi(C)\cong \mathscr{H}^{M-1}(\Phi(C))[-M+1] \oplus \mathscr{H}^{M}(\Phi(C))[-M]$ as above. We will show that the two diagrams 
$$\xymatrix{
\mathscr{H}^{M}(\Phi(C))[-M] \ar[r] \ar[d] & \Phi(A)[1] \ar[d]\\
F(C) \ar[r] & F(A)[1]
}
\qquad \text{ and } \qquad
\xymatrix{
\mathscr{H}^{M-1}(\Phi(C))[-M+1] \ar[r] \ar[d] & \Phi(A)[1] \ar[d]\\
F(C) \ar[r] & F(A)[1]
}
$$
are both commutative. 

Notice that the composition $\Phi(B)\to \Phi(C) \to F(C) \to F(A)[1]$ is zero, because we already know that the central square in (\ref{triangle}) commutes, and $F(B) \to F(C) \to F(A)[1]$ is zero. Moreover, since $\Phi(B)\to  \mathscr{H}^{M-1}(\Phi(C))[-M+1]$ is the zero map, this means that the composition $\Phi(B)\to \mathscr{H}^{M}(\Phi(C))[-M]\to F(C) \to F(A)[1]$ is zero. Since we already know that on objects we have an isomorphism $F(\cdot)\cong \Phi(\cdot)$, we can repeat the same computation as above and get that
$$\text{Hom}(\mathscr{H}^{M}(\Phi(C))[-M], F(A)[1]) \to \text{Hom}(\Phi(B), F(A)[1])$$
is again injective hence the composition $\mathscr{H}^{M}(\Phi(C))[-M]\to F(C)\to F(A)[1]$ is zero. In the same way, we know that $\mathscr{H}^{M}(\Phi(C))[-M]\to \Phi(A) \to F(A)[1]$ is also zero. This shows that the first square commutes.

To show that the second square above is commutative, we just need to show that the square
$$\xymatrix{
\mathscr{H}^{M-1}(\Phi(C))[-M+1] \ar[r] \ar[d] & \Phi(A)[1] \ar[d]\\
\mathscr{H}^{M-1}(F(C))[-M+1] \ar[r] & F(A)[1]
}$$
is commutative. But this follows from Proposition \ref{isoforfree}.

\textbf{VIII. $s(-)$ is compatible with triangles of the type  $0\to A\to B\to C\to 0$ for any $A$ and $B$:} in this situation we can find $A'$, $B'$ locally free and a diagram
$$\xymatrix{
0\ar[r]	&A'\ar[r]\ar[d]		&B'\ar[r]\ar[d]		& C\ar[r]\ar@{=}[d] & 0 \\
0\ar[r]	&A\ar[r]			&B\ar[r]			& C\ar[r]				& 0 \mathrm{.}
}$$
Then we get
$$\xymatrix{
\Phi(A)\ar[r]					&\Phi(B)\ar[r]					& \Phi(C)\ar[r]				& \Phi(A)[1] \ar@/^20mm/[ddd]^{s(A)[1]}					\\	
\Phi(A')\ar[r]\ar[d]\ar[u]		&\Phi(B')\ar[r]\ar[d]\ar[u]		& \Phi(C)\ar[r]\ar[d]^{\hskip .3in{\circlearrowleft}} \ar@{=}[u]_{\hskip .3in{\circlearrowleft}} 	&  \Phi(A')[1]\ar[d]^{\hskip .3in{\circlearrowleft}}\ar[u]	\\
F(A')\ar[r]\ar[d]				&F(B')\ar[r]\ar[d]				& F(C)\ar[r]\ar@{=}[d]^{\hskip .3in{\circlearrowleft}} 	&  F(A')[1]\ar[d]\\
F(A)\ar[r]					&F(B)\ar[r]					& F(C)\ar[r]				& F(A)[1] 
}$$
where the top and bottom right squares commute because $\Phi$ and $F$ are functors, the middle right square by part VII, and the semi-circle by part III. Therefore the boundary maps commute:
$$\xymatrix{
\Phi(C)\ar[r]\ar[d]^{s(C)} &\Phi(A)[1] \ar[d]^{s(A)[1]}	\\
 F(C)\ar[r] & F(A)[1]\mathrm{.}
}$$
\vskip-.5cm
\end{proof}

\begin{lemma}\label{stillzero}
In the setup of Theorem \ref{onebla}, let $A$, $B$ be two torsion coherent sheaves on $X$. Consider a coherent sheaf $A'$ on $X$ with a surjection $A'\to A$. Consider a map 
$$\xi:\Phi(A)\to F(B)$$
that induces the zero map on all cohomology groups. If the composition $\Phi(A')\to \Phi(A)\to F(B)$ is zero, then it follows that the map $\xi$ is zero.
\end{lemma}

\begin{proof}
We know that
$$\Phi(A)\cong  p_{2*}(\text{Tor}^{1}(\mathscr{Bla}^{M},p_{1}^{*}A))[-M+1]\oplus p_{2*}(\mathscr{Bla}^{M}\otimes p_{1}^{*}A)[-M]$$
(since $\mathscr{Bla}^{M}$ is supported at a finite number of points and hence is flasque). Moreover, we know that $F(B)$ is isomorphic to $\Phi(B)$ by part III.  of Theorem \ref{onebla}, so we also know that
$$F(B)\cong  p_{2*}(\text{Tor}^{1}(\mathscr{Bla}^{M},p_{1}^{*}B))[-M+1]\oplus p_{2*}(\mathscr{Bla}^{M}\otimes p_{1}^{*}B)[-M]$$

Fix two isomorphisms as above. Now if we know that the given map $\Phi(A)\to F(B)$ is zero on cohomology, the map can be represented by a map
$$p_{2*}(\mathscr{Bla}^{M}\otimes p_{1}^{*}A)[-M] \to  p_{2*}(\text{Tor}^{1}(\mathscr{Bla}^{M},p_{1}^{*}B))[-M+1]$$
i.e. an element of $\text{Ext}^{1}(p_{2*}(\mathscr{Bla}^{M}\otimes p_{1}^{*}A), p_{2*}(\text{Tor}^{1}(\mathscr{Bla}^{M},p_{1}^{*}B)))$. Then it suffices to show that the map 
$$\text{Ext}^{1}(p_{2*}(\mathscr{Bla}^{M}\otimes p_{1}^{*}A), p_{2*}(\text{Tor}^{1}(\mathscr{Bla}^{M},p_{1}^{*}B))) \to \text{Ext}^{1}(p_{2*}(\mathscr{Bla}^{M}\otimes p_{1}^{*}A'), p_{2*}(\text{Tor}^{1}(\mathscr{Bla}^{M},p_{1}^{*}B)))$$
is injective.

But since $A'$ surjects onto $A$, we have a surjection $p_{2*}(A'\otimes \mathscr{Bla}^{M})\to p_{2*}(A\otimes \mathscr{Bla}^{M})$ and both of these sheaves are supported at the points $q_{1},\ldots,q_{t}$, hence this is a surjection of vector spaces and therefore it splits. Hence the map on $\text{Ext}^{1}$ above is injective.
\end{proof}

To complete the proof of Theorem \ref{generaliso} we still need to extend the isomorphism to the whole derived category $D^{b}_{\mathrm{Coh}}(X)$. This is straightforward in our case of $\text{dim}(X)=1$ thanks to the following:

\begin{lemma}\label{generality}
Let $X$, $Y$ be smooth projective varieties with $\mathrm{dim}(X)=1$. Consider two exact functors  $F,\Phi:D^{b}_{\mathrm{Coh}}(X)\to D^{b}_{\mathrm{Coh}}(Y)$, and assume that there exists an isomorphism of exact functors $s:\Phi\to F$ on the full subcategory of $D^{b}_{\mathrm{Coh}}(X)$ given by coherent sheaves on $X$ placed in degree zero. Then $s$ extends to an isomorphism of exact functors on the whole $D^{b}_{\mathrm{Coh}}(X)$.
\end{lemma}

\begin{proof}
Consider a complex $C^{\bullet} \in D^{b}_{\mathrm{Coh}}(X)$. Then by \cite{dold} $C^{\bullet}\cong \oplus H^{i}(C^{\bullet})[-i]$, in a non-canonical way. Choose one such isomorphism for each $C^{\bullet}$. By Theorem \ref{onebla}, since  both functors are compatible with shifting, we immediately get an isomorphism $s(C^{\bullet}):\Phi(C^{\bullet})\to F(C^{\bullet})$.

Now consider a map $C^{\bullet} \to D^{\bullet}$. This is the same as a map $ \oplus H^{i}(C^{\bullet})[-i]\to  \oplus H^{i}(D^{\bullet})[-i] $, and again since the two functors are compatible with shifting, and $X$ has dimension 1, it is enough to show that $s(-)$ is compatible with maps $\F\to \G$ and $\F\to \G[1]$, where $\F$ and $\G$ are sheaves. The first case follows from the fact that $s$ is an isomorphism of exact functors on $\mathrm{Coh}(X)$. A map $\alpha: \F \to \G[1]$ corresponds to an element in $\text{Ext}^{1}(\F, \G)$ so we have a short exact sequence
$$0\to \G \to \mathscr{H} \to \F \to 0$$
and by Theorem \ref{onebla} we get an isomorphism of triangles 
$$\xymatrix{
\Phi(\G) \ar[r] \ar[d]^{s(\G)}& \Phi(\mathscr{H}) \ar[r] \ar[d]^{s(\mathscr{H})} &\Phi(\F) \ar[d]^{s(\F)} \ar[r]^-{\Phi(\alpha)} & \Phi(\G)[1]	\ar[d]^{s(\G)[1]}\\
\F(\G) \ar[r] 				& F(\mathscr{H}) \ar[r]					&F(\F) \ar[r]^-{F(\alpha)}					& F(\G)[1]
}$$
hence $s$ is compatible with $\alpha$. The fact that $s$ is compatible with triangles is immediate.
\end{proof}

\begin{proof}[Proof of Theorem \ref{generaliso}]
This follows immediately from Theorem \ref{onebla} and Lemma \ref{generality}.
\end{proof}


\begin{remark}
Notice that any functor satisfying the hypotheses of Theorem \ref{generaliso} will not be full and will not satisfy
$$\mathrm{Hom}_{D^{b}_{\mathrm{Coh}}(Y)} (F(\mathscr{F},\mathscr{G}[j])=0 \text{ if } j<0$$ 
for all $\mathscr{F}, \mathscr{G} \in \Ol_{X}$ (take for example $\mathscr{F}$ to be supported at one of the $p_{i}$'s). Hence this improves the result of \cite{twisted}.
\end{remark}

\section{A Spectral Sequence}\label{ss}

Even when we don't know how to build a kernel out of the sheaves $\mathscr{Bla}^{i}$ that we constructed in Theorem \ref{all blas}, these sheaves still satisfy some good properties. As an example, we will show that under the same hypotheses of Theorem \ref{all blas} the analogue of the Cartan-Eilenberg Spectral Sequence converges when the dimension of $X$ is one, whereas the sequence of low degree terms is exact for any $X$, $Y$.

Let $X$, $Y$ be smooth projective varieties over an algebraically closed fiend, and consider a Fourier-Mukai functor $\Phi_{E}$ with $E\in D^{b}_{\mathrm{Coh}}(X\times Y)$. Then for each locally free sheaf $\mathscr{E}\in \text{Coh}(X)$ the Cartan-Eilenberg Spectral Sequence gives
$$E_{2}^{pq}=R^{p}p_{2*}(\mathscr{H}^{q}(E)\otimes p_{1}^{*} \mathscr{E}) \Rightarrow \mathscr{H}^{p+q}(\Phi_{E}(\mathscr{E}))$$

Now consider an exact functor $F: D^{b}_{\mathrm{Coh}}(X) \to D^{b}_{\mathrm{Coh}}(Y)$. Suppose we computed the cohomology sheaves $\mathscr{Bla}^{i}$ of the prospective kernel in $D^{b}_{\mathrm{Coh}}(X\times Y)$ as in Theorem \ref{all blas}. Then we can replace $\mathscr{H}^{q}(E)$ with $\mathscr{B}^{q}$ in the above and set 
$$E_{2}^{pq}=R^{p}p_{2*}(\mathscr{B}^{q}\otimes p_{1}^{*} \mathscr{E})$$
for any locally free sheaf $\mathscr{E}$ on $X$. The corresponding sequence of low degree terms is exact:

\begin{proposition}\label{ss1}
Let $X$, $Y$ be smooth projective varieties over an algebraically closed field $k$, $F: D^{b}_{\mathrm{Coh}}(X) \to D^{b}_{\mathrm{Coh}}(Y)$ an exact functor. Assume that $F(\mathscr{E})\in D^{[M,N]}_{Coh}(Y)$ for all locally free sheaves $\mathscr{E}$ of finite rank on $X$. Let $\mathscr{Bla}^{i}$ be the sheaves computed in Theorem \ref{all blas}. Then for any locally free sheaf $\mathscr{E}$ of finite rank on $X$ the following sequence is exact:
\begin{align*}
0 &\to R^{1}p_{2*}(\mathscr{Bla}^{M}\otimes p_{1}^{*} \mathscr{E}) \to \mathscr{H}^{M+1}(F(\mathscr{E})) \to\\
&\to p_{2*}(\mathscr{Bla}^{M+1}\otimes p_{1}^{*} \mathscr{E})  \to R^{2}p_{2*}(\mathscr{Bla}^{M}\otimes p_{1}^{*} \mathscr{E}) \to \mathscr{H}^{M+2}(F(\mathscr{E})) 
\end{align*}
\end{proposition}

\begin{proof}
Assume that there is an embedding $X\to \pr^{d}_{k}$. Then for every $m>0$ we have a short exact sequence
$0\to \Ol_{X}\to \Ol_{X}(m)^{\oplus(d+1)}\to K_{m}\to 0$, where $K_{m}$ is a locally free sheaf. 

Let $\mathscr{E}$ be a locally free sheaf of finite rank on $X$. Then by tensoring the sequence above with $\mathscr{E}$ we get a short exact sequence
$$0\to \mathscr{E}\to \mathscr{E}(m)^{\oplus(d+1)}\to K_{m}\otimes \mathscr{E}\to 0\mathrm{.}$$
Choose $m$ high enough so that $R^{1}p_{2*}(\mathscr{B}^{M}\otimes p_{1}^{*}(\mathscr{E}(m)))=0$ and that $\mathscr{H}^{M+1}(F(\mathscr{E}(m)))\cong p_{2*}(\mathscr{B}^{M+1}\otimes p_{1}^{*}\mathscr{E})$. By applying the functor $F$ and then taking cohomology we get a long exact sequence
$$0\to \mathscr{H}^{M}(F(\mathscr{E})) \to  \mathscr{H}^{M}(F(\mathscr{E}(m)))^{\oplus(d+1)} \to  \mathscr{H}^{M}(F(K_{m}\otimes \mathscr{E})) \to  \mathscr{H}^{M+1}(F(\mathscr{E}))\to\ldots$$
By Proposition \ref{bla_for_free}, for any locally free sheaf of finite rank $\mathscr{F}$ we have a functorial isomorphism
$$\mathscr{H}^{M}(F(\mathscr{F})) \xrightarrow{\cong} p_{2*}(\mathscr{B}^{M}\otimes p_{1}^{*}\mathscr{F})\mathrm{.}$$
Then we get the following diagram:
$$\xymatrix@C=5mm{
\ldots \ar[r] & \mathscr{H}^{M}(F(\mathscr{E}(m)))^{\oplus(d+1)} \ar[r]\ar[d]^{\cong}  &			\mathscr{H}^{M}(F(K_{m}\otimes \mathscr{E})) \ar[r]\ar[d]^{\cong} &  \mathscr{H}^{M+1}(F(\mathscr{E})) \ar[r] &\ldots\\
\ldots\ar[r] &p_{2*}(\mathscr{B}^{M}\otimes p_{1}^{*}\mathscr{E}(m)^{\oplus(d+1)}) \ar[r] & p_{2*}(\mathscr{B}^{M}\otimes p_{1}^{*}(K_{m}\otimes\mathscr{E})) \ar[r]  & R^{1}p_{2*}(\mathscr{B}^{M}\otimes p_{1}^{*}\mathscr{E}) \ar[r] & 0
}$$
so there exists a map
$$R^{1}p_{2*}(\mathscr{B}^{M}\otimes p_{1}^{*}\mathscr{E}) \to \mathscr{H}^{M+1}(F(\mathscr{E}))\mathrm{.} $$
By Theorem \ref{all blas} we also have a map $\mathscr{H}^{M+1}(F(\mathscr{E})) \to p_{2*}(\mathscr{B}^{M+1}\otimes p_{1}^{*}\mathscr{E})$. The fact that the sequence 
$$0\to R^{1}p_{2*}(\mathscr{B}^{M}\otimes p_{1}^{*}\mathscr{E}) \to \mathscr{H}^{M+1}(F(\mathscr{E})) \to p_{2*}(\mathscr{B}^{M+1}\otimes p_{1}^{*}\mathscr{E})$$
is exact follows from diagram chasing. This is the first part of our sequence.

Now since the sequence above is exact for any $\mathscr{E}$ locally free sheaf of finite rank on $X$, it will also be exact for $K_{m}\otimes \mathscr{E}$. So we have the following diagram:
\vskip-5mm
$$\xymatrix@C=3mm@R=5mm{& & & 0\ar[d]\\
 & &  & R^{1}p_{2*}(\mathscr{Bla}^{M}\otimes p_{1}^{*}(K_{m}\otimes \mathscr{E})) \ar[d] \\
\ldots \ar[r] &  \mathscr{H}^{M+1}(F(\mathscr{E})) \ar[r]\ar[d] & \mathscr{H}^{M+1}(F(\mathscr{E}(m)))^{\oplus (d+1)} \ar[r]\ar[d]^{\cong} &\mathscr{H}^{M+1}(F(K_{m}\otimes \mathscr{E}))\ar[d] \ar[r] &\ldots \\
\ldots \ar[r] & p_{2*}(\mathscr{Bla}^{M+1}\otimes p_{1}^{*}\mathscr{E}) \ar[r] &  p_{2*}(\mathscr{Bla}^{M+1}\otimes p_{1}^{*}\mathscr{E}(m)^{\oplus(d+1)}) \ar[r] & p_{2*}(\mathscr{Bla}^{M+1}\otimes p_{1}^{*}(K_{m}\otimes \mathscr{E})) \ar[r] &\ldots
}$$
by diagram chasing we get a map 
$$p_{2*}(\mathscr{Bla}^{M+1}\otimes p_{1}^{*}\mathscr{E}) \to R^{1}p_{2*}(\mathscr{Bla}^{M}\otimes p_{1}^{*}(K_{m}\otimes \mathscr{E}))$$
This has an obvious map to $\mathscr{H}^{M+2}(F(\mathscr{E}))$ given by the composition 
$$R^{1}p_{2*}(\mathscr{Bla}^{M}\otimes p_{1}^{*}(K_{m}\otimes \mathscr{E})) \to \mathscr{H}^{M+1}(F(K_{m}\otimes \mathscr{E})) \to \mathscr{H}^{M+2}(F(\mathscr{E}))$$
but since $R^{1}p_{2*}(\mathscr{Bla}^{M}\otimes p_{1}^{*} \mathscr{E}(m))=R^{2}p_{2*}(\mathscr{Bla}^{M}\otimes p_{1}^{*} \mathscr{E}(m))=0$, we know that
$$R^{1}p_{2*}(\mathscr{Bla}^{M}\otimes p_{1}^{*}(K_{m}\otimes \mathscr{E}))\cong R^{2}p_{2*}(\mathscr{Bla}^{M}\otimes p_{1}^{*}\mathscr{E})$$
this gives the second part of our sequence,
$$p_{2*}(\mathscr{Bla}^{M+1}\otimes p_{1}^{*}\mathscr{E}) \to R^{2}p_{2*}(\mathscr{Bla}^{M}\otimes p_{1}^{*}\mathscr{E}) \to \mathscr{H}^{M+2}(F(\mathscr{E}))$$
Exactness of the whole sequence again follows by diagram chasing.
\end{proof}

\begin{proposition}
In the setting of Proposition \ref{ss1}, assume $\mathrm{dim}(X)=1$. Then for all locally free sheaves of finite rank $\mathscr{E}$ on $X$ there is a spectral sequence
$$E_{2}^{pq}=R^{p}p_{2*}(\mathscr{Bla}^{q}\otimes p_{1}^{*} \mathscr{E})\Rightarrow \mathscr{H}^{p+q}(F(\mathscr{E}))$$
\end{proposition}

\begin{proof}
Since dim $X=1$, the only nonzero terms of the spectral sequence are $E_{2}^{0,q}$ and $E_{2}^{1,q}$. Therefore all the differentials are zero and to show that the spectral sequence converges we need to show:
\begin{itemize}
\item There exists a map  $E_{2}^{1,q-1} = R^{1}p_{2*}(\mathscr{Bla}^{q-1}\otimes p_{1}^{*} \mathscr{E}) \hookrightarrow \mathscr{H}^{q}(F(\mathscr{E}))$
\item $E_{2}^{0,q}= p_{2*}(\mathscr{Bla}^{q}\otimes p_{1}^{*} \mathscr{E}) \cong \mathscr{H}^{q}(F(\mathscr{E}))/R^{1}p_{2*}(\mathscr{Bla}^{q-1}\otimes p_{1}^{*} \mathscr{E})$.
\end{itemize}

Since dim $X=1$ we have $R^{2}p_{2*}(\mathscr{Bla}^{q}\otimes p_{1}^{*} \mathscr{E})=0$. Therefore the exact sequence of Proposition \ref{ss1} becomes a short exact sequence
\begin{equation}\label{ss2}
0\to R^{1}p_{2*}(\mathscr{Bla}^{M}\otimes p_{1}^{*} \mathscr{E}) \to \mathscr{H}^{M+1}(F(\mathscr{E})) \to p_{2*}(\mathscr{Bla}^{M+1}\otimes p_{1}^{*} \mathscr{E})  \to 0\mathrm{.}
\end{equation}
Choose $m$ high enough so that $R^{p}p_{2*}(\mathscr{B}^{q}\otimes p_{1}^{*}(\mathscr{E}(m)))=0$ for all $q$ and all $p>0$ (this can be done because for $m$ high enough $p_{1}^{*}\Ol_{X}(1)$ is very ample with respect to $X\times Y\to Y$), and such that $ \mathscr{H}^{i}(F(\mathscr{E}(m)))\cong p_{2*}(\mathscr{Bla}^{i}\otimes p_{1}^{*}\mathscr{E}(m))$ for all $i$ (this can be done by Theorem \ref{all blas}). Then using again the short exact sequence in the proof of Proposition \ref{ss1}
$$0\to \mathscr{E}\to \mathscr{E}(m)^{\oplus (d+1)} \to K_{m}\otimes \mathscr{E}\to 0\mathrm{,}$$
we get that
$$R^{1}p_{2*}(\mathscr{Bla}^{i}\otimes p_{1}^{*} (\mathscr{E}\otimes K_{m})) \cong R^{2}p_{2*}(\mathscr{Bla}^{i}\otimes p_{1}^{*} \mathscr{E})=0$$
for all $i$.

Now assume by induction that we get the same short exact sequence as (\ref{ss2}) starting with $\mathscr{Bla}^{M+n-1}$ for any locally free sheaf of finite rank $\mathscr{E}$:
$$0\to R^{1}p_{2*}(\mathscr{Bla}^{M+n-1}\otimes p_{1}^{*} \mathscr{E}) \to \mathscr{H}^{M+n}(F(\mathscr{E})) \to p_{2*}(\mathscr{Bla}^{M+n}\otimes p_{1}^{*} \mathscr{E})  \to 0$$
then the same exact sequence will hold if we substitute $\mathscr{E}$ with $K_{m}\otimes \mathscr{E}$:
$$0\to R^{1}p_{2*}(\mathscr{Bla}^{M+n-1}\otimes p_{1}^{*} (K_{m}\otimes \mathscr{E})) \to \mathscr{H}^{M+n}(F(K_{m}\otimes \mathscr{E})) \to p_{2*}(\mathscr{Bla}^{M+n}\otimes p_{1}^{*} (K_{m}\otimes \mathscr{E}))  \to 0\mathrm{.}$$
But since $R^{1}p_{2*}(\mathscr{Bla}^{i}\otimes p_{1}^{*} (\mathscr{E}\otimes K_{m}))=0$ this gives an isomorphism 
$$ \mathscr{H}^{M+n}(F(K_{m}\otimes \mathscr{E})) \cong p_{2*}(\mathscr{Bla}^{M+n}\otimes p_{1}^{*} (K_{m}\otimes \mathscr{E}))\mathrm{.} $$
Hence from the diagram
$$\xymatrix@C=2mm{ 
\ldots \ar[r] & \mathscr{H}^{M+n}(F(K_{m}\otimes \mathscr{E})) \ar[rrrr]\ar[d]^{\cong} & &&&\mathscr{H}^{M+n+1}(F(\mathscr{E})) \ar[d]\ar[r] &\ldots \\
\ldots\ar[r] & p_{2*}(\mathscr{Bla}^{M+n}\otimes p_{1}^{*}(K_{m}\otimes \mathscr{E})) \ar[r] & R^{1}p_{2*}(\mathscr{Bla}^{M+n}\otimes p_{1}^{*}\mathscr{E}) \ar[r] &0&0\ar[r] & p_{2*}(\mathscr{Bla}^{M+n+1}\otimes p_{1}^{*}\mathscr{E})\ar[r]& \ldots
}$$
we get a sequence
\begin{equation}\label{surjective}
0\to R^{1}p_{2*}(\mathscr{Bla}^{M+n}\otimes p_{1}^{*} \mathscr{E}) \to \mathscr{H}^{M+n+1}(F(\mathscr{E}))\to p_{2*}(\mathscr{Bla}^{M+n+1}\otimes p_{1}^{*} \mathscr{E})\mathrm{,}
\end{equation}
which is exact by diagram chasing. Again, we also have the corresponding exact sequence for the locally free sheaf $K_{m}\otimes \mathscr{E}$:
$$0\to R^{1}p_{2*}(\mathscr{Bla}^{M+n}\otimes p_{1}^{*} (K_{m}\otimes\mathscr{E})) \to \mathscr{H}^{M+n+1}(F(K_{m}\otimes\mathscr{E}))\to p_{2*}(\mathscr{Bla}^{M+n+1}\otimes p_{1}^{*} (K_{m}\otimes \mathscr{E}))$$
and the first term of the sequence is zero, i.e. the map $\mathscr{H}^{M+n+1}(F(K_{m}\otimes\mathscr{E}))\to p_{2*}(\mathscr{Bla}^{M+n+1}\otimes p_{1}^{*} (K_{m}\otimes \mathscr{E}))$ is injective. This is reflected in the following diagram:
$$\xymatrix@C=1mm{ 
\ldots \ar[r] & \mathscr{H}^{M+n+1}(F(\mathscr{E})) \ar[r]\ar[d] & \mathscr{H}^{M+n+1}(F(\mathscr{E}(m)))^{\oplus (d+1)}\ar[d]^{\cong}\ar[r] & \mathscr{H}^{M+n+1}(F(K_{m}\otimes \mathscr{E})) \ar[r]\ar@{^{(}->}[d] &\ldots   \\
0\ar[r] & p_{2*}(\mathscr{Bla}^{M+n+1}\otimes p_{1}^{*}\mathscr{E}) \ar[r] & p_{2*}(\mathscr{Bla}^{M+n+1}\otimes p_{1}^{*}\mathscr{E}(m)^{\oplus (d+1)})\ar[r]  & 
p_{2*}(\mathscr{Bla}^{M+n+1}\otimes p_{1}^{*}(K_{m}\otimes \mathscr{E})) \ar[r] & \ldots
}$$
By diagram chasing this tells us that the map 
$$\mathscr{H}^{M+n+1}(F(\mathscr{E}))\to p_{2*}(\mathscr{Bla}^{M+n+1}\otimes p_{1}^{*} \mathscr{E})$$
is actually surjective, hence (\ref{surjective}) becomes a short exact sequence, and this completes the proof.
\end{proof}

\bibliography{refs}{}
\bibliographystyle{amsalpha}

\end{document}